\documentclass[a4paper, leqno, 12pt]{article}
\makeindex

\usepackage{amsmath, amssymb, amsthm}

\usepackage[all]{xy}

    \newcommand{\cT}{{\mathcal{T}}}

    \newcommand{\cO}{ {\mathcal{O} } }
    
    \newcommand{\T}{{\mathcal{T}}}

    \renewcommand{\mod}{{\;\rm mod}}

    \newcommand{\be}{\begin{equation}}
    \newcommand{\ee}{\end{equation}}

\newtheorem{thm}{Theorem}[section]
\newtheorem{lemma}[thm]{Lemma}
\newtheorem{cor}[thm]{Corollary}
\newtheorem{prop}[thm]{Proposition}
\newtheorem{theorem}[thm]{Theorem}

\theoremstyle{definition}

\theoremstyle{remark}
\newtheorem{rem}[thm]{Remark}
\newtheorem{example}[thm]{Example}

\newcommand{\ZZ}{{\mathbb Z}}
\newcommand{\CC}{{\mathbb C}}
\newcommand{\FF}{{\mathbb F}}

\newcommand{\NN}{{\mathbb N}}

\newcommand{\PP}{{\mathbb P}}

\newcommand{\oplusm}{\mathop{\bigoplus}\limits}
\newcommand{\summ}{\mathop{\sum}\limits}

\begin{document}
    \bibliographystyle{amsplain}

\title{
Quadratic Differentials and Equivariant Deformation Theory of
Curves}

\author{B.~K\"ock and A.~Kontogeorgis}

    \maketitle

    \begin{abstract}
Given a finite $p$-group $G$ acting on a smooth projective
curve $X$ over an algebraically closed field $k$ of
characteristic $p$, the dimension of the tangent space of the
associated equivariant deformation functor is equal to the
dimension of the space of coinvariants of $G$ acting on the
space $V$ of global holomorphic quadratic differentials on $X$.
We apply known results about the Galois module structure of
Riemann-Roch spaces to compute this dimension when $G$ is
cyclic or when the action of $G$ on $X$ is weakly ramified.
Moreover we determine certain subrepresentations of $V$, called
$p$-rank representations.
        \end{abstract}

    \section{Introduction}

Let $X$ be a non-singular complete curve of genus $g_X \ge 2$
defined over an algebraically closed field $k$ of
characteristic $p$ and let $G$ be a subgroup of the
automorphism group of $X$. The equivariant deformation problem
associated with this situation is to determine in how many ways
$X$ can be deformed to another curve that also allows $G$ as a
group of automorphisms. More precisely, we are led to study the
deformation functor $D$ which maps any local Artinian
$k$-algebra $A$ to the set of isomorphism classes of lifts of
$(G,X)$ to an action of $G$ on a smooth scheme $\tilde{X}$ over
$\textrm{Spec}(A)$, see \cite{Be-Me} for a detailed description
of $D$. The dimension of the tangent space of $D$ (in the sense
of Schlessinger \cite{Schlessinger}) can be thought of as a
(crude) answer to the above-mentioned equivariant deformation
problem. In \cite{Be-Me}, Bertin and M\'ezard have shown that
the tangent space of $D$ is isomorphic to the equivariant
cohomology $H^1(G, \cT_X)$ of $(G,X)$ with values in the
tangent sheaf $\cT_X$ of $X$. A more detailed analysis of the
above-mentioned equivariant deformation problem also involves
determining obstructions in $H^2(G, \cT_X)$ and studying the
structure of the versal deformation ring associated with $D$,
see \cite{Be-Me} and \cite{CK}.

In the classical case, i.e.\ when $k=\CC$ (see
e.g.~\cite[Section~V.2.2]{Farkas-Kra}) or, more generally, when
the action of $G$ on $X$ is tame (see \cite[p.~206]{Be-Me} and
\cite[formula~(40)]{Ko:2002a}), the dimension of the tangent
space of $D$ is known to be equal to $3g_Y-3+r$ where $g_Y$
denotes the genus of the quotient curve $Y= X/G$ and $r$
denotes the cardinality of the branch locus $Y_\textrm{ram}$;
note that $3g_Y-3$ is the dimension of the moduli space of
curves of genus $g_Y$ and each branch point adds one further
degree of freedom in our deformation problem, as expected.

In the case of wild ramification, the computation of the
dimension of the tangent space of $D$ turns out to be a
difficult problem. In this paper, we disregard the known case
of tame ramification and assume that the characteristic $p$ of
$k$ is positive and that $G$ is a $p$-group.

Formulas for the dimension of $H^1(G, \cT_X)$ can be derived
from results in \cite{Be-Me} and \cite{CK} when in addition $G$
is a cyclic group  or $X$ is an ordinary curve (or, more
generally, the action of $G$ on $X$ is weakly ramified), see
Remarks~\ref{secondproof} and~\ref{remarkweakly}. Furthermore,
when $G$ is an elementary abelian group, the second-named
author computes this dimension in \cite{Ko:2002a} using the
Lyndon-Hochschild-Serre spectral sequence.

The object of this paper is to pursue the following alternative
method for computing the dimension of $H^1(G,\cT_X)$. In
\cite{KoJPAA06}, the second-named author has shown that this
dimension is equal to the dimension of the space $H^0(X,
\Omega_X^{\otimes 2})_G$ of coinvariants of $G$ acting on the
space of global holomorphic quadratic differentials on $X$. We
will use known results on the Galois module structure of
Riemann-Roch spaces $H^0(X,\cO_X(D))$ for certain $G$-invariant
divisors $D$ on $X$ to determine the dimension of
$H^0(X,\Omega_X^{\otimes 2})_G$.

We now fix some notations. Let $P_1, \ldots, P_r$ denote a
complete set of representatives for the $G$-orbits of ramified
points. The corresponding decomposition groups, ramification
indices and coefficients of the ramification divisor are
denoted by $G(P_j)$, $e_0(P_j)$ and $d(P_j)$, $j=1, \ldots, r$,
respectively.

Section~\ref{cyliccase} of this paper will deal with the case
when $G$ is cyclic. Generalizing a result of Nakajima
\cite[Theorem~1]{Nak:86} from the case when $G$ is cyclic of
order~$p$ to the case when $G$ is an arbitrary cyclic
$p$-group, Borne \cite[Theorem~7.23]{Borne06} has explicitly
determined the Galois module structure of $H^0(X, \cO_X(D))$
for any $G$-invariant divisor $D$ on $X$ of degree greater than
$2g_X-2$. Although the formulation of Borne's theorem requires
quite involved definitions and is in particular difficult to
state, its consequence for the dimension of
$H^0(X,\Omega_X^{\otimes 2})_G$ is simple:
\[\dim_k H^0(X, \Omega_X^{\otimes 2})_G = 3g_Y-3 + \sum_{j=1}^r
\left\lfloor \frac{2 d(P_j)}{e_0(P_j)} \right\rfloor,\] see
Corollary~\ref{dimensioncoinvariants}. This result can also be
derived from core results in \cite{Be-Me} and becomes
\cite[Proposition~4.1.1]{Be-Me} if $G$ is cyclic.

In Section~\ref{sec3} we assume that the action of $G$ on $X$
is weakly ramified, i.e.\ that the second ramification group
$G_2(P)$ vanishes for all $P \in X$, and prove the formula
\[\dim_k H^0(X, \Omega_X^{\otimes 2})_G = 3g_Y-3 + \sum_{j=1}^r
\log_p |G(P_j)| +
\begin{cases}
2r & \textrm{if } p >3\\
r & \textrm{if } p=2 \textrm{ or } 3,
\end{cases} \]
see Theorem~3.1. This formula matches up with results of
Cornelissen and Kato \cite[Theorem~4.5 and Theorem~5.1(b)]{CK}
for the deformation of ordinary curves, see
Remark~\ref{remarkweakly}. The proof of our result proceeds
along the following lines. Let $W$ denote the space of global
meromorphic quadratic differentials on $X$ which may have a
pole of order at most~$3$ at each ramified point and which are
holomorphic everywhere else. A result of the first-named author
\cite[Theorem~4.5]{Koeck:04} implies that $W$ is a free
$k[G]$-module and that its rank and hence the dimension of
$W_G$ is equal to $3(g_Y-1+r)$, see
Proposition~\ref{projective}. The difference between this
dimension and the dimension of $H^0(X, \Omega_X^{\otimes 2})_G$
can be expressed in terms of group homology of $G$ acting on
the space of global sections of a certain skyscraper sheaf and
can be computed using a spectral sequence argument, see
Proposition~\ref{grouphomology} and
Lemma~\ref{homologicalalgebra}.

In Section~\ref{pranksection} we first show that there exists
an {\em effective} $G$-invariant canonical divisor $D$ on $X$,
see Lemma~\ref{efecDG}. This allows us to split the
$k[G]$-module $H^0(X, \Omega_X^{\otimes 2}) \cong H^0(X,
\Omega_X(D))$ into its semisimple and nilpotent part with
respect to the corresponding Cartier operator \cite{Stalder04}:
\[H^0(X, \Omega_X(D)) \cong H_D^\mathrm{s} \oplus
H_D^\mathrm{n}.\] Little seems to be known about the
$k[G]$-module structure of the nilpotent part $H_D^\mathrm{n}$.
We use a result of Nakajima \cite[Theorem~1]{Nak:85} to show
that the semisimple part $H_D^\mathrm{s}$ is a free
$k[G]$-module and to compute its rank and hence the dimension
of $(H_D^\mathrm{s})_G$ provided $D$ satisfies further
conditions, see Theorem~\ref{semisimple}. In cases the
$k[G]$-module structure of $H^0(X, \Omega_X^{\otimes 2})$ is
known (such as Section~\ref{cyliccase}) we then also get
information about $H_D^\mathrm{n}$, see
Corollary~\ref{directsummand}.

Section~\ref{account} is an appendix and gives an account of a
structure theorem (Theorem~\ref{structuretheorem}) for weakly
ramified Galois extensions of local fields. This structure
theorem is used in Section~\ref{sec3}, but only in the case
$p=2$. We finally derive a feature of the Weierstrass semigroup
at any ramified point of $X$ if $p=2$ and the action of $G$ on
$X$ is weakly ramified, see Corollary~\ref{Weierstrass}.

{\bf Acknowledgments:} The second-named author would like  to
thank Gunther Cornelissen  and  Michel Matignon for pointing
him to the work of Niels Borne, Bernhard K\"ock and Nicolas
Stalder.

{\bf Notations.} Throughout this paper $X$ is a connected
smooth projective curve over an algebraically closed field $k$
of positive characteristic $p$ and $G$ is a finite subgroup of
the automorphism group of $X$ whose order is a power of $p$. We
also assume that the genus $g_X$ of $X$ is at least~$2$. For
any point $P \in X$ let $G(P)= \{g \in G: g(P) =P\}$ denote the
decomposition group and $e_0(P) = |G(P)|$ the ramification
index at $P$. The $i^\mathrm{th}$ ramification group at $P$ in
lower notation  is
\[G_i(P) = \{g \in
G(P): g(t_P)-t_P \in (t_P^{i+1})\}\] where $t_P$ is a local
parameter at $P$. Let $e_i(P)$ denote its order. Note that
$e_0(P) = e_1(P)$ because $G$ is a $p$-group. Let
\[\pi: X \rightarrow Y:= X/G\]
denote the canonical projection, and let $g_Y$ denote the genus
of $Y$. We use the notation $X_\mathrm{ram}$ for the set of
ramification points and $Y_{\rm ram} = \pi(X_\mathrm{ram})$ for
the set of branch points of $\pi$. We fix a complete set $P_1,
\ldots, P_r$ of representatives for the $G$-orbits of ramified
points in $X$; in particular we have $r=|Y_{\rm ram}|$. The
sheaves of differentials and of relative differentials on $X$
are denoted by $\Omega_X$ and $\Omega_{X/Y}$, respectively.
Furthermore for any divisor $D$ on $X$ we use the notation
$\Omega_X(D)$ for the sheaf $\Omega_X \otimes_{\mathcal{O}_X}
\mathcal{O}_X(D)$. If $D= \sum_{P\in X} n_P [P]$ is an
effective divisor, we write
\[D_{\rm red} := \sum_{P\in X: \; n_P \not= 0} [P]\] for the
associated reduced divisor. Let
\[R= \sum_{P \in X} d(P) [P]\]
denote the ramification divisor of $\pi$; then $R_{\rm red} =
\sum_{P \in X_{\rm ram}} [P]$ is the reduced ramification
divisor. Note that
\[d(P) = \sum_{i=0}^\infty (e_i(P) -1)\]
by Hilbert's different formula \cite[Theorem~3.8.7]{StiBo}. The
notation $K_Y$ stands for a canonical divisor on $Y$; the
divisor $K_X:= \pi^* K_Y + R$ is then a $G$-invariant canonical
divisor on $X$ by \cite[Proposition~IV~2.3]{Hartshorne:77}. For
any real number $x$, the notation $\lfloor x \rfloor$ means the
largest integer less than or equal to $x$.

\section{The Cyclic Case}\label{cyliccase}

In this section we assume that the group $G$ is cyclic of order
$p^\nu$ and compute the dimension $\dim_k H^1(G, \cT_X) =
\dim_k H^0(X,\Omega_X^{\otimes 2})_G$ of the tangent space of
the deformation functor associated with $G$ acting on $X$, see
Corollary~\ref{dimensioncoinvariants}. We derive
Corollary~\ref{dimensioncoinvariants} from
Theorem~\ref{numberofindecomposables} which in turn is a
consequence of a theorem of Borne \cite[Theorem~7.23]{Borne06}.
While a lot of notations and explanations are needed to
formulate Borne's very fine statement,
Theorem~\ref{numberofindecomposables} is easy to state and
sufficient to derive the important
Corollary~\ref{dimensioncoinvariants}. A different way of
formulating Corollary~\ref{dimensioncoinvariants} is given in
Lemma~\ref{HasseArf} and a different way of proving it is
outlined in Remark~\ref{secondproof}.

Let $\sigma$ denote a generator of $G$. Let $V$ denote the
$k[G]$-module with $k$-basis $e_1, \ldots, e_{p^\nu}$ and
$G$-action given by $\sigma(e_1)= e_1$ and $\sigma(e_l) = e_l +
e_{l-1}$ for $l= 2, \ldots, p^\nu$. It is well-known that the
submodules $V_l := \mathrm{Span}_k(e_1, \ldots, e_l)$, $l= 1,
\ldots, p^\nu$, of $V$ form a set of representatives for the
set of isomorphism classes of indecomposable $k[G]$-modules. In
particular for every finitely generated $k[G]$-module $M$ there
are some integers $m_l(M) \ge 0$, $l=1, \ldots, p^\nu$, such
that
\[M \cong \bigoplus_{l=1}^{p^\nu} \; V_l^{\bigoplus m_l(M)};\]
let $\mathrm{Tot}(M) = \sum_{l=1}^{p^\nu} m_l(M)$ denote the
total number of direct summands.

\begin{theorem}\label{numberofindecomposables}
Let $D= \sum_{P \in X} n_P [P]$ be a $G$-invariant divisor on
$X$ such that $\mathrm{deg}(D) > 2g_X -2$. Then we have
\[\mathrm{Tot}(H^0(X, \cO_X(D))) = 1- g_Y + \sum_{Q \in Y} \left\lfloor
\frac{n_{\tilde{Q}}}
{e_0({\tilde{Q}})} \right\rfloor \] where ${\tilde{Q}}$
denotes any point in the fiber $\pi^{-1}(Q)$.
\end{theorem}

\begin{rem}
Our global assumption that $g_X \ge 2$ is not necessary for
this theorem.
\end{rem}

\begin{proof}
Borne defines certain divisors $D(l)$, $l=1, \ldots, p^\nu$, on
$Y$ and proves that
\[m_l(H^0(X, \cO_X(D))) = \mathrm{deg}(D(l)) - \mathrm{deg}(D(l+1)) \quad
\textrm{for } l = 1, \ldots, p^\nu -1\] and
\[m_{p^\nu}(H^0(X,\cO_X(D))) = 1- g_Y + \mathrm{deg}(D(p^\nu)),\]
see \cite[Theorem~7.23]{Borne06}. (The reader who wants to
match this with \cite{Borne06} may find useful to observe that
$(p-1) + (p-1) p + \ldots + (p-1)p^{\nu-1}$ is the $p$-adic
expansion of $p^\nu - 1$.) We therefore have
\[\mathrm{Tot}(H^0(X,\cO_X(D))) = \sum_{l=1}^{p^\nu}
m_l(H^0(X,\cO_X(D))) = \mathrm{deg}(D(1)) + 1 - g_Y.\] We now
explain the definition of $D(1)$. For each $\mu=0, \ldots, \nu$
let $H_\mu$ denote the unique subgroup of $G$ of order $p^\mu$
and set $X_\mu := X/H_\mu$. Furthermore let $\pi_\mu : X_{\mu
-1} \rightarrow X_\mu$ denote the canonical projection. For
each divisor $E$ on $X_{\mu - 1}$ set $(\pi_\mu)_*^0(E) :=
\left\lfloor \frac{1}{p}(\pi_\mu)_*(E)\right\rfloor$ where
$\left\lfloor \cdot \right\rfloor$ denotes the integral part of
a divisor, taken coefficient by coefficient. Then $D(1)$ is
defined as
\[D(1) = (\pi_\nu)_*^0 \ldots (\pi_1)_*^0 (D),\]
see \cite[Definition~7.22 and Theorem~7.23]{Borne06}. Now
Theorem~\ref{numberofindecomposables} follows from the formula
\[D(1) = \sum_{Q\in Y} \left\lfloor
\frac{n_{\tilde{Q}}}{e_0({\tilde{Q}})}\right\rfloor [Q].\] We prove this
formula by induction on $\nu$. It is obvious if $\nu = 0$. So
let $\nu \ge 1$. By the inductive hypothesis we have
\[(\pi_{\nu -1})_*^0 \ldots (\pi_1)_*^0(D) = \sum_{R \in X_{\nu -1}}
\left\lfloor \frac{n_{\tilde{R}}}{|G_0(\tilde{R}) \cap H_{\nu
-1}|} \right\rfloor [R]\] where, as above, $\tilde{R}$ denotes any
point in the fiber $(\pi_{\nu -1} \circ \ldots \circ
\pi_1)^{-1}(R)$. Thus, if $Q \in Y$ is not a branch point of $\pi_\nu$,
the coefficient of the divisor $D(1)$ at $Q$ is equal to
\[\left\lfloor \frac{n_{\tilde{Q}}}{|G_0({\tilde{Q}}) \cap H_{\nu -
1}|}
\right\rfloor = \left\lfloor \frac{n_{{\tilde{Q}}}} {|G_0({\tilde{Q}})|}
\right\rfloor = \left\lfloor \frac{n_{\tilde{Q}}}{e_0({\tilde{Q}})}
\right\rfloor,\] as desired. Otherwise we have $|G_0({\tilde{Q}})
\cap H_{\nu -1} | = |H_{\nu -1}| = p^{\nu -1}$ and the
coefficient of $D(1)$ at $Q$ is equal to
\[\left\lfloor \frac{\left\lfloor \frac{n_{\tilde{Q}}}{p^{\nu -1}}
\right\rfloor}{p} \right\rfloor\,\, \stackrel{(*)}{=} \left\lfloor
\frac{n_{\tilde{Q}}}{p^\nu} \right\rfloor = \left\lfloor
\frac{n_{\tilde{Q}}}{e_0({\tilde{Q}})} \right\rfloor, \] again as
stated. (To see $(*)$ write $n_{\tilde{Q}}$ in the form $ap^\nu
+ b p^{\nu -1} + c$ with some $a \ge 0$, $b \in \{0, \ldots,
p-1\}$ and $c \in \{0, \ldots, p^{\nu -1}\}$.)
\end{proof}

\begin{cor}\label{dimensioncoinvariants}
We have
\[\dim_k H^0(X, \Omega_X^{\otimes 2})_G = 3g_Y -3 + \sum_{j=1}^r
\left\lfloor \frac{2d(P_j)}{e_0(P_j)}\right\rfloor.\]
\end{cor}

\begin{proof}
Let $K_Y = \sum_{Q \in Y} m_Q [Q]$ be a canonical divisor on
$Y$ and let $K_X$ denote the canonical $G$-invariant divisor
$\pi^*K_Y + R$ on $X$. As $\dim (V_l)_G = 1$ for all $l = 1,
\ldots, p^\nu$ we have
\begin{align*}
\dim_k H^0(X, &\Omega_X^{\otimes 2})_G = \mathrm{Tot}(H^0(X, \Omega_X^{\otimes 2}))\\
& = \mathrm{Tot}(H^0(X, \cO_X(2K_X)))
=\mathrm{Tot} (H^0(X, \cO_X(2 \pi^* K_Y + 2R))).
\end{align*}
For any $Q \in Y$ the coefficient of the divisor $2\pi^*K_Y +
2R$ at any point $\tilde{Q} \in \pi^{-1}(Q)$ is
$2e_0(\tilde{Q}) m_Q + 2 d(\tilde{Q})$. Thus from
Theorem~\ref{numberofindecomposables} we obtain
\begin{eqnarray*}
\lefteqn{\dim_k H^0(X, \Omega_X^{\otimes 2})_G}\\
&=& 1- g_Y + 2 \left(\sum_{Q \in Y} m_Q\right) + \sum_{Q \in Y}
\left\lfloor \frac{2d(\tilde{Q})}{e_0(\tilde{Q})} \right\rfloor \\
&=& 3g_Y-3 + \sum_{j=1}^r \left\lfloor \frac{2d(P_j)}{e_0(P_j)} \right\rfloor,
\end{eqnarray*}
as stated.
\end{proof}

\begin{rem}\label{secondproof}
The following chain of equalities sketches a different approach
to prove Corollary~\ref{dimensioncoinvariants} based on core
results of the paper \cite{Be-Me} by Bertin and M\'ezard.
\begin{eqnarray*}
\lefteqn{\dim_k H^0(X, \Omega_X^{\otimes 2})_G}\\
&=& \dim_k H^1(G, \cT_X)\\
&=& \dim_k H^1(Y, \pi_*^G(\cT_X)) + \sum_{j=1}^r \dim_k H^1(G(P_j), \hat{\cT}_{X, P_j})\\
&=& 3g-3 + \sum_{j=1}^r \left\lceil \frac{d(P_j)}{e_0(P_j)} \right\rceil +
\sum_{j=1}^r \left(\left\lfloor \frac{2d(P_j)}{e_0(P_j)} \right\rfloor - \left\lceil \frac{d(P_j)}{e_0(P_j)} \right\rceil \right)\\
&=&3g_Y-3 + \sum_{j=1}^r \left\lfloor \frac{2d(P_j)}{e_0(P_j)} \right\rfloor.
\end{eqnarray*}
The statements (together with some notations) used in the above
chain of equalities can be found in \cite{KoJPAA06},
\cite[pp.~205-206]{Be-Me}, \cite[formula~(40)]{Ko:2002a},
\cite[Proposition~4.1.1]{Be-Me}.
\end{rem}

The following lemma gives a different interpretation of the
term $\lfloor \frac{2d(P_j)}{e_0(P_j)} \rfloor$ in
Corollary~\ref{dimensioncoinvariants}. To simplify notation we
fix one of the points $P_1, \ldots, P_r$ and write just $P$ for
this point and just $d$ and $e_0, e_1, \ldots$ for $d(P)$ and
$e_0(P)$, $e_1(P), \ldots$, respectively. Furthermore, $N$ and
$M$ denote the highest jumps in the lower ramification
filtration and in the upper ramification filtration of $G(P)$,
respectively.

\begin{lemma}\label{HasseArf}
We have
\[\left\lfloor \frac{2d}{e_0} \right\rfloor = 2(1+M) +
\left\lfloor \frac{-2(1+N)}{e_0} \right\rfloor.\]
\end{lemma}

\begin{proof}
Let $k:= \log_p e_0$. By the Hasse-Arf theorem (see the example
on page~76 in \cite{SeL}) there exist positive integers $a_0,
a_1, \ldots, a_{k-1}$ so that the sequence of jumps in the
lower ramification filtration is
\[i_1 = a_o, \quad i_2 = a_0 + pa_1, \quad \ldots, \quad i_k = a_0 + pa_1 +
\ldots + p^{k-1} a_{k-1}\] and the sequence of jumps in the
upper ramification filtration is
\[a_0, \quad a_0+a_1, \quad \ldots, \quad a_0 + \ldots + a_{k-1}.\]
We therefore obtain
\begin{eqnarray*}
\lefteqn{d= \sum_{i=0}^{i_1} (e_{i_1} - 1) +
\sum_{i=i_1 + 1}^{i_2} (e_{i_2} -1) + \ldots +
\sum_{i=i_{k-1} +1}^{i_k} (e_{i_k} -1)}\\
&=& \sum_{i=0}^{i_1}(p^k-1) +
\sum_{i=i_1 +1}^{i_2}(p^{k-1} -1) + \ldots +
\sum_{i=i_{k-1} + 1}^{i_k} (p-1)\\
&=& (1+i_1)(p^k-1) + (i_2 - i_1)(p^{k-1} - 1) + \ldots +
(i_k - i_{k-1})(p-1)\\
&=& (1+a_0)(p^k -1) + (pa_1)(p^{k-1} -1) + \ldots + (p^{k-1} a_{k-1})(p-1)\\
&=& (1+ a_0 + \ldots + a_{k-1}) p^{k} - (1+a_0+ pa_1 + \ldots + p^{k-1} a_{k-1})\\
&=& (1+M)p^k - (1+N)
\end{eqnarray*}
and hence
\[\left\lfloor \frac{2d}{e_0} \right\rfloor =
\left\lfloor \frac{2(1+M)p^k - 2(1+N)}{p^k} \right\rfloor =
2(1+M) + \left\lfloor \frac{-2(1+N)}{e_0} \right\rfloor,\] as stated.
\end{proof}

\section{The Weakly Ramified Case} \label{sec3}

In this section we assume that the cover $\pi: X \rightarrow Y$
is weakly ramified, i.e.\ that $G_i(P)$ is trivial for all $P
\in X$ and all $i \ge 2$, and prove the following explicit
formula for the dimension $\dim_k H^1(G, \T_X) = \dim_k H^0(X,
\Omega_X^{\otimes 2})_G$ of the tangent space of the
deformation functor associated with $G$ acting on~$X$.

\begin{theorem}\label{weaklythm}
We have
\[\dim_k H^0(X, \Omega_X^{\otimes 2})_G = 3g_Y-3 + \summ_{j=1}^r
\log_p|G(P_j)| +
\begin{cases}
2r & \textrm{if } p >3\\
r  & \textrm{if } p=2 \textrm{ or } 3.
\end{cases}\]
\end{theorem}

\begin{rem}\label{remarkweakly}
Notice, that if the curve $X$ is ordinary, then the cover $\pi$
is weakly ramified \cite[Theorem~2(i)]{Nak}. In this case,
Theorem~\ref{weaklythm} can be proved, similarly to
Remark~\ref{secondproof}, using a result of G.~Cornelissen and
F.~Kato on deformations of ordinary curves
\cite[Theorem~4.5]{CK}. In fact, the arguments used in the
proof of \cite[Theorem~4.5]{CK} also work in the weakly
ramified case and we thus obtain an alternative proof of
Theorem~\ref{weaklythm} in the general case.
\end{rem}

\begin{example}
If we moreover assume that $G$ is cyclic then  the group
$G(P_j)$ is cyclic of order $p$ for all $j=1, \ldots, r$ and
the formula in Theorem~\ref{weaklythm} becomes
\[\dim_k H^0(X, \Omega_X^{\otimes 2})_G = 3g_Y-3 +
\begin{cases}
3r & \textrm{if } p >3\\
2r & \textrm{if } p =2 \textrm{ or } p=3,
\end{cases}\]
which is the same as in Corollary~\ref{dimensioncoinvariants}.
\end{example}

\begin{proof}[Proof (of Theorem~\ref{weaklythm})]
Let $\Sigma$ denote the skyscraper sheaf defined by the short
exact sequence
\[0 \rightarrow \Omega_X^{\otimes 2} \rightarrow
\Omega_X^{\otimes 2}(3 R_{\rm red}) \rightarrow \Sigma
\rightarrow 0.\] Since $\mathrm{deg}(\Omega_X^{\otimes 2}) =
2(2 g_X-2) > 2g_X -2$, we have $H^1(X, \Omega_X^{\otimes 2})=0$
(\cite[Example~IV~1.3.4]{Hartshorne:77}). By applying the
functor of global sections we therefore obtain the short exact
sequence
\[0 \rightarrow H^0(X, \Omega_X^{\otimes 2}) \rightarrow H^0(X,
\Omega_X^{\otimes 2}(3 R_{\rm red})) \rightarrow H^0(X, \Sigma)
\rightarrow 0.\] As the $k[G]$-module $H^0(X, \Omega_X^{\otimes
2} (3 R_{\rm red}))$ is projective (by
Proposition~\ref{projective} below), its higher group homology
vanishes. From the long exact group-homology sequence
associated with the previous short exact sequence we thus
obtain the exact sequence
\begin{align*}0
\rightarrow H_1(G,H^0(X, \Sigma)) \rightarrow &H^0(X,
\Omega_X^{\otimes 2})_G \\ &\rightarrow H^0(X, \Omega_X^{\otimes
2}(3 R_{\rm red}))_G \rightarrow H^0(X, \Sigma)_G \rightarrow
0.
\end{align*}
Using a result by the first-named author on the $k[G]$-module
structure of Riemann-Roch spaces in the weakly ramified case we
will show in Proposition~\ref{projective} below that the
$k[G]$-module $H^0(X, \Omega_X^{\otimes 2}(3R_{\rm red}))$ is
free and then derive the dimension of $H^0(X,\Omega_X^{\otimes
2}(3R_{\rm red}))_G$. Furthermore we will explicitly describe
the $k[G]$-module structure of $H^0(X, \Sigma)$ and determine
(the difference between) the dimension of $H^0(X, \Sigma)_G$
and of $H_1(G,H^0(X,\Sigma))$ using (unfortunately rather
involved) homological computations (see
Proposition~\ref{grouphomology} and
Lemma~\ref{homologicalalgebra}). Theorem~\ref{weaklythm} then
immediately follows from the formulas obtained in
Propositions~\ref{projective} and~\ref{grouphomology}.
\end{proof}

\begin{prop}\label{projective}
The $k[G]$-module $H^0(X, \Omega_X^{\otimes 2}(3 R_{\rm red}))$
is a free $k[G]$-module of rank $3(g_Y-1+r)$. In particular we
have
\[\dim_k H^0(X,\Omega_X^{\otimes 2}(3 R_{\rm red}))_G = 3 (g_Y-1+r).\]
\end{prop}

\begin{proof}
We will first show that the $k[G]$-module $H^0(X,
\Omega_X^{\otimes 2}(3 R_{\rm red}))$ is free. As $G$ is a
$p$-group it suffices to show that it is projective. Let $D =
\sum_{P\in X} n_P [P]$ be a $G$-invariant divisor on $X$.
Theorem~2.1 in \cite{Koeck:04} tells us that the Riemann-Roch
space $H^0(X, \mathcal{O}_X(D))$ is projective whenever both
$H^1(X,\mathcal{O}_X(D)) = 0$ and $n_P \equiv -1 \mod \;
e_1(P)$ for all $P \in X_{\rm ram}$. It therefore suffices to
check these two conditions for the divisor $D=2K_X + 3R_{\rm
red}$. The first condition follows from
\cite[Example~IV~1.3.4]{Hartshorne:77} because $\mathrm{deg}(D)
\ge 2(2g_X-2)> 2g_X-2$. The second condition follows from the
formulas $K_X = \pi^*K_Y +R$ and
\begin{equation}\label{ramdivisor}
R=\sum_{P \in X_{\rm ram}} 2(e_0(P)-1) [P].
\end{equation}
Thus $H^0(X, \Omega_X^{\otimes 2}(3 R_{\rm red}))$ is a free
$k[G]$-module. We now determine its rank. As $H^1(X,
\Omega_X^{\otimes 2}(3 R_{\rm red}))$ vanishes (see above) we
have
\[\dim_k H^0(X, \Omega_X^{\otimes 2}(3 R_{\rm red})) = 2
(2g_X-2) + 3 |X_{\rm ram}| + 1 - g_X = 3(g_X - 1 + |X_{\rm
ram}|)\] by the Riemann-Roch theorem
\cite[Theorem~IV~1.3]{Hartshorne:77}. Furthermore the
Rie\-mann-Hurwitz formula
\cite[Corollary~IV~2.4]{Hartshorne:77} and
formula~(\ref{ramdivisor}) imply that
\[2g_X-2 = |G| (2g_Y-2) + \sum_{P \in X_{\rm ram}} 2
(e_0(P)-1).\] By combining the previous two equations we obtain
\begin{eqnarray*}
\lefteqn{\dim_k H^0(X, \Omega_X^{\otimes 2}(3 R_{\rm red}))}\\
& = & 3 \left(|G|(g_Y-1) + \sum_{P \in X_{\rm ram}} (e_0(P) -1) + |X_{\rm
ram}|\right) \\
&= &3|G|(g_Y-1 +r).
\end{eqnarray*}
In particular the rank of $H^0(X,\Omega_X^{\otimes 2}(3 R_{\rm
red}))$ over $k[G]$ is $3(g_Y-1+r)$ and
\[\dim_k(H^0(X, \Omega_X^{\otimes 2}(3 R_{\rm red}))_G = 3
(g_Y-1 +r),\] as stated.
\end{proof}

\begin{rem}
A slightly different approach to Proposition~\ref{projective}
is to first check the two conditions for the divisor $D= 2 K_X
+ 3 R_\mathrm{red}$ as above but then to use Theorem~4.5 in
\cite{Koeck:04} which computes the isomorphism class of $H^0(X,
\cO_X(D))$ directly.
\end{rem}

\begin{prop}\label{grouphomology}
If $p >3$ we have
\[\dim_k H_q(G, H^0(X, \Sigma)) =
\begin{cases}
r & \textrm{for } q=0\\
\summ_{j=1}^r \log_p|G(P_j)| & \textrm{for } q=1.
\end{cases}\]
If $p=3$ we have
\[\dim_k H_q(G, H^0(X, \Sigma)) =
\begin{cases}
r & \textrm{for } q=0\\
\summ_{j=1}^r (\log_3|G(P_j)|-1) & \textrm{for } q=1.
\end{cases}\]
If $p=2$ we have
\[\dim_k H_0(G,H^0(X,\Sigma)) - \dim_k H_1(G, H^0(X, \Sigma)) =
2r-\sum_{j=1}^r \log_2|G(P_j)|.\]
\end{prop}

\begin{proof}
As $\Sigma$ is a skyscraper sheaf, $H^0(X, \Sigma)$ is the
direct sum of the stalks of $\Sigma$:
\[H^0(X,\Sigma) \cong \bigoplus_{P \in X_{\rm ram}} \Sigma_P
\cong \bigoplus_{j=1}^r \mathrm{Ind}_{G(P_j)}^G (\Sigma_{P_j})
.\]  We therefore have
\[H_q(G,H^0(X, \Sigma)) \cong \bigoplus_{j=1}^r H_q(G(P_j),
\Sigma_{P_i}) \qquad \textrm{ for } q\ge 0\] by Shapiro's lemma
\cite[6.3.2, p.171]{Weibel}. We now fix one of the points $P_1,
\ldots, P_r$ and write just $P$ for this point. Let $t$ and $s$
be local parameters at $P$ and $\pi(P)$, respectively. Let $m$
denote the multiplicity of a canonical divisor $K_Y$ at
$\pi(P)$. By formula~(\ref{ramdivisor}) the multiplicity of the
canonical divisor $K_X = \pi^* K_Y +R$ at $P$ is then equal to
$me_0(P) + 2e_0(P) -2$. Hence the multiplicity of $2K_X +
3R_{\rm red}$ is equal to $2me_0(P) + 4e_0(P) -1$, and we
obtain
\[\Sigma_P \cong (t^{1-e_0(P)(2m+4)})/(t^{4-e_0(P)(2m+4)}) \cong
(t)/(t^4)\] where the latter isomorphism is given by
multiplication with the $G(P)$-invariant element $s^{2m+4}$. We
have $g(t) \equiv t$ mod $(t^2)$ for all $g \in G(P)$, since
$G(P) = G_1(P)$. Let the maps $\alpha$ and $\beta$ from $G(P)$
to $k$ be defined by the congruences
\begin{equation}\label{congruence}
g(t) \equiv t + \alpha(g) t^2 + \beta(g) t^3 \quad \mathrm{mod} \quad (t^4),
\end{equation}
$g \in G(P)$. The map $\alpha$ is obviously a homomorphism from
$G(P)$ to the additive group of $k$ and it is injective because
$G_2(P)$ is trivial. In particular $G(P)$ is a non-trivial
elementary abelian $p$-group. Let $\omega_1: = t^3 \mod \;
(t^4)$, $\omega_2 := t^2 \mod \; (t^4)$ and $\omega_3:= t \mod
\; (t^4)$. The congruences~(\ref{congruence}) imply that
\begin{eqnarray}
g (\omega_1) & =& \omega_1 \label{groupaction1}\\
g (\omega_2) & =& \omega_2  + 2 \alpha(g)\omega_1 \label{groupaction2}\\
g (\omega_3) & =& \omega_3  + \alpha(g) \omega_2 +\beta(g) \omega_1 \label{groupaction3}
\end{eqnarray}
for any $g \in G(P)$. Furthermore we easily derive that the map
$\beta: G \rightarrow k$ satisfies the condition
\begin{equation} \label{beta}
\beta(hg) = \beta(h) + 2 \alpha(h) \alpha(g) +
\beta(g)
\end{equation}
for all $h,g \in G(P)$. Now Proposition~\ref{grouphomology}
follows from the following homo\-logical-algebra lemma which we
formulate in a way that is independent from the context of this
paper. In the case $p=2$ we need to moreover use
Proposition~\ref{structuretheorem} which tells us that we can
choose the local parameter $t$ in such a way that $\beta =
\alpha^2$; in particular $\beta$ is not a $k$-multiple of
$\alpha$ unless $G(P)$ is cyclic in which case $3-2\log_2|G(P)|
= 1 = 2- 2 \log_2|G(P)|$.
\end{proof}

\begin{lemma}\label{homologicalalgebra}
Let $k$ be a field of characteristic $p
>0$. Let $G$ be a non-trivial elementary abelian $p$-group and
let $\alpha$ and $\beta$ be maps from $G$ to $k$. We assume
that $\alpha$ is a non-zero homomorphism and that $\beta$
satisfies the condition~(\ref{beta}) for all $h,g \in G$.
Furthermore let $V$ be a vector space over $k$ with basis
$\omega_1$, $\omega_2$, $\omega_3$. We assume that $G$ acts on
$V$ by $k$-automorphisms given by (\ref{groupaction1}),
(\ref{groupaction2}) and~(\ref{groupaction3}) (for any $g \in
G$). If $p >3$, we have
\[\dim_kH_q(G, V) =
\begin{cases}
1 & \textrm{for } q=0\\
\log_p |G| & \textrm{for } q=1.
\end{cases}\]
If $p=3$, we have
\[\dim_kH_q(G, V) =
\begin{cases}
1 & \textrm{for } q=0\\
\log_3 |G| -1 & \textrm{for } q=1.
\end{cases}\]
If $p=2$, we have
\begin{align*}\dim_k H_0(G, V) - &\dim_k H_1(G,V)\\
&=
\begin{cases}
3-2\log_2|G| & \textrm{if } \beta = c \alpha \textrm{ for some } c \in k,\\
2-\log_2|G| & \textrm{else.}
\end{cases}
\end{align*}
\end{lemma}

\begin{proof}
Let $s$ denote the dimension of $G$ when viewed as vector space
over $\FF_p$, i.e.\ $s=\log_p|G|$, and let $g_1, \ldots, g_s$
be a basis of $G$ over $\FF_p$ such that $\alpha(g_1) \not= 0$.
For $i =1 ,\ldots, s$ the sequence
\begin{equation*}
\ldots \; \stackrel{g_i-1}{\longrightarrow } \; k[\langle g_i \rangle]
\;\stackrel{1+g_i+\ldots+g_i^{p-1}}{\longrightarrow}\;
k[\langle g_i \rangle] \; \stackrel{g_i-1}{\longrightarrow} \;
k[\langle g_i \rangle] \; \stackrel{\rm sum}{\longrightarrow}
\; k \longrightarrow 0
\end{equation*}
is a $k[\langle g_i\rangle]$-projective resolution of the
trivial $k[\langle g_i \rangle]$-module~$k$, see
\cite[Section~6.2]{Weibel}. By the K\"unneth formula
\cite[Theorem~3.6.3]{Weibel}, the tensor product of these
sequences is a $k[G]$-projective resolution of the trivial
$k[G]$-module~$k$:
\begin{equation}\label{resolution}
\ldots \rightarrow \left(\bigoplus_{i=1}^s k[G] \right)
\bigoplus \left(\bigoplus_{i < j} k[G]\right) \;
\stackrel{e}{\rightarrow} \; \bigoplus_{i=1}^s k[G]
\;\stackrel{d}{\rightarrow}\; k[G] \; \stackrel{\rm
sum}{\rightarrow}\; k \rightarrow 0.
\end{equation}
Here the differentials $d$ and $e$ are given as follows: $d$
acts on the $i^{\rm th}$ direct summand of $\bigoplus_{i=1}^s
k[G]$ by multiplication with $g_i -1$; $e$ maps the $i^{\rm
th}$ direct summand of $\bigoplus_{i=1}^s k[G]$ to the $i^{\rm
th}$ direct summand of $\bigoplus_{i=1}^s k[G]$ by
multiplication with $1+g_i + \ldots + g_i^{p-1}$ and it maps
the direct summand of $\bigoplus_{i<j} k[G]$ indexed by the
pair $(i,j)$ to the direct sum of the two direct summands of
$\bigoplus_{i=1}^s k[G]$ indexed by $i$ and $j$ by
multiplication with $g_j-1$ and $1-g_i$, respectively (note
that we have here adopted the iterated sign trick explained in
\cite[Theorem~3.6.3]{Weibel}). By tensoring the
$k[G]$-projective resolution~(\ref{resolution}) with $V$ over
$k[G]$ we obtain the complex
\[\ldots \rightarrow \left(\bigoplus_{i=1}^s V\right) \bigoplus
\left(\bigoplus_{i<j} V \right) \; \rightarrow
\bigoplus_{i=1}^s V \rightarrow V\] which we denote by $C.$ (with $V$ sitting in the place $0$) and
whose homology modules are $H.(G,V)$ by definition. We have the
filtration
\[V_0 := \{0\} \subset V_1 := \mathrm{Span}_k(\omega_1)
\subset V_2:= \mathrm{Span}_k(\omega_1, \omega_2) \subset
V_3:=V\] of $V$ by $G$-stable subspaces of $V$. By tensoring
the $k[G]$-projective resolution~(\ref{resolution}) of $k$ with
$V_i$ for $i=0, \ldots, 3$ over $k[G]$ we obtain the filtration
\[0=F_0C. \subset F_1C. \subset F_2 C. \subset F_3C. = C.\]
of the complex $C.$ by subcomplexes. Let
\[{E^1_\mathbf{pq} = H_\mathbf{p+q}(F_\mathbf{p}C./F_\mathbf{p-1}C.) \Rightarrow H_\mathbf{p+q}(C.)}\]
denote the convergent spectral sequence associated with this
filtered complex (\cite[Theorem~5.5.1]{Weibel}). (To avoid
confusion with the prime $p$ we use $\mathbf{p}$ rather than
$p$ to denote the index in the spectral sequence.) Since $G$
acts trivially on $V_\mathbf{p}/V_\mathbf{p-1}$ and since the
characteristic of $k$ is $p$, both the maps $g_i-1$ and $1+g_i+
\ldots + g_i^{p-1}$ act as the zero map on
$V_\mathbf{p}/V_\mathbf{p-1}$ for all $\mathbf{p}$. Hence all
the differentials in $F_\mathbf{p}C./F_\mathbf{p-1}C.$ are zero
and the $E^1$-page of our spectral sequence looks as follows
(with $V_1$ sitting at the place $(1,-1)$):
\[ \begin{array}{ccccc}
\vdots &  & \vdots &  & \vdots \\
\oplusm_{i=1}^s V_1 & \stackrel{\partial_4}{\leftarrow} & \left(\oplusm_{i=1}^s V_2/V_1 \right)
\oplusm \left( \oplusm_{i<j} V_2/V_1 \right) & \leftarrow &\vdots\\ \\
V_1 & \stackrel{\partial_3}{\leftarrow} & \oplusm_{i=1}^s V_2/V_1 & \stackrel{\partial_2}{\leftarrow} &
\left(\oplusm_{i=1}^s V_3/V_2 \right) \oplusm \left(\oplusm_{i<j} V_3/V_2\right)\\ \\
0 & \leftarrow & V_2/V_1 & \stackrel{\partial_1}{\leftarrow} & \oplusm_{i=1}^s V_3/V_2\\ \\
0 & \leftarrow & 0 &\leftarrow & V_3/V_2
\end{array}\]
As all the columns indexed by $p \le 0$ or $p \ge 4$ are zero
we obtain the equalities
\[\dim_k H_0(G,V) = \dim_k E^1_{3,-3} + \dim_k E^2_{2,-2} + \dim_k E^3_{1,-1}\] and
\[\dim_k H_1(G,V) = \dim_k E^3_{3,-2} + \dim_k E^2_{2,-1} +
\dim_k E^3_{1,0}.\] We now assume that $p >3$ and prove
Lemma~\ref{homologicalalgebra} by showing that the six
dimensions on the right hand side are equal to $1$, $0$, $0$,
$s-1$, $0$ and $1$, respectively. It is obvious that the
(first) dimension $\dim_k E^1_{3,-3}= \dim_k V_3/V_2$ is equal
to 1. To determine the remaining five dimensions, let
$\partial_1$, $\partial_2$, $\partial_3$ and $\partial_4$
denote the differentials as indicated in the above diagram of
the $E^1$-page of our spectral sequence. We are now going to
explicitly describe these differentials by using the fact that
they are connecting homomorphisms associated with the short
exact sequences of complexes
\[0 \rightarrow F_1 C. \rightarrow F_2 C. \rightarrow F_2 C./
F_1 C. \rightarrow 0 \] and
\[0 \rightarrow F_2 C./ F_1 C. \rightarrow F_3 C./ F_1 C. \rightarrow
F_3C./F_2 C. \rightarrow 0.\] For any $a_1, \ldots, a_s \in k$
we have
\begin{eqnarray*}
\lefteqn{\partial_1((a_1 \bar{\omega}_3,
\ldots, a_s \bar{\omega}_3))}\\
& =& \overline{(g_1 -1)(a_1
\omega_3)} + \ldots + \overline{(g_s-1)(a_3 \omega_3)}\\
&= &(a_1
\alpha(g_1) + \ldots + a_r \alpha(g_s))
\bar{\omega}_2.
\end{eqnarray*}
In particular, the differential $\partial_1$ is surjective
(since the homomorphism $\alpha$ is non-zero) and we obtain
\[\dim_k E^2_{2,-2} = 0,\] as claimed above, and
\[\dim_k E^2_{3,-2} = s-1.\]
Similarly, for any $a_1, \ldots, a_s \in k$, we have
\[\partial_3((a_1\bar{\omega}_2, \ldots, a_s \bar{\omega}_2)) =
2 ((a_1 \alpha(g_1) + \ldots + a_s\alpha(g_s)) \bar{\omega}_1.\]
In particular $\partial_3$ is surjective as well and we obtain
\[\dim_k E^3_{1,-1} = \dim_k E^2_{1,-1} = 0,\] as claimed above;
hence the differential from $E^2_{3,-2}$ to $E^2_{1,-1}$ is
zero and we conclude
\[\dim_k E^3_{3,-2} = \dim_k E^2_{3,-2} = s-1,\]
as claimed above. As $\alpha(g_1)\not= 0$, the $s-1$ tuples
\begin{eqnarray*}
y_1 & := & (\alpha(g_2) \bar{\omega}_2, - \alpha(g_1) \bar{\omega}_2, 0, \ldots, 0)\\
y_2 & := & (\alpha(g_3) \bar{\omega}_2, 0, -\alpha(g_1) \bar{\omega}_2, 0, \ldots, 0)\\
& \vdots & \\
y_{s-1} & := & (\alpha(g_s) \bar{\omega}_2,  0 , \ldots, 0, - \alpha(g_1) \bar{\omega}_2).
\end{eqnarray*}
of $\bigoplus_{i=1}^s V_2/V_1$ are linearly independent over
$k$ and (hence) span the kernel of~$\partial_3$. For $i=1,
\ldots, s-1$ let $x_i$ be the tuple of $\Big(\bigoplus_{i=1}^s
V_3/V_2\Big) \bigoplus \left(\bigoplus_{i<j} V_3/V_2 \right)$
that has $\bar{\omega}_3$ at the place $(1, i+1)$ and $0$
everywhere else. From the description of the differential $e$
given above we obtain
\[\partial_2(x_i) = y_i \quad \textrm{ for } i=1, \ldots,
s-1.\] In particular the image of $\partial_2$ is equal to the
kernel of $\partial_3$ and we conclude
\[\dim_k E^2_{2,-1}=0,\]
as claimed above. Similarly we obtain that the dimension of the
image of $\partial_4$ is $s-1$ and hence
\[\dim_k E^2_{1,0} = \dim_k (\mathrm{coker}(\partial_4)) = 1.\]
To prove that also $\dim_k E^3_{1,0}=1$ (as claimed above) we
will show that the differential
\[\partial: E^2_{3,-1} = \mathrm{ker}(\partial_2)
\rightarrow \mathrm{coker}(\partial_4) = E^2_{1,0}\] is zero.
This differential is defined as follows, see
\cite[Section~5.4]{Weibel}. Let $X \in E^2_{3,-1} =
\mathrm{ker}(\partial_2: H_2 (F_3 C./ F_2 C.) \rightarrow
H_1(F_2 C./F_1 C.))$. We write $X$ as the residue class of some
$x \in F_3 C_2$. Then $e(x) \equiv e(y) \mod \; F_1 C_1$ for
some $y \in F_2 C_2$ and $\partial (X)$ is equal to the residue
class of $e(x-y)$ in $H_1(F_1 C.)$. As all the differentials in
$F_2 C./ F_1 C.$ are zero (see above) we may choose $y=0$; in
particular $e(x) \in F_1 C_1$ and $\partial(X)$ is equal to the
residue class of $e(x) \in \bigoplus_{i=1}^s V_1$ modulo
$\mathrm{im}(\partial_4)$. More precisely, let $x = (\tilde{x},
\hat{x})$ with $\tilde{x} \in \bigoplus_{i=1}^s V_3$ and
$\hat{x} \in \bigoplus_{i<j} V_3$. From formula~(\ref{beta})
and the fact that $\alpha$ is a homomorphism we obtain
\[\beta(g^m) = m \beta(g) + \binom{m}{2} 2 \alpha(g)^2\]
for any $g \in G$ and $m\ge 1$. For $i=1, \ldots, s$ we
therefore have
\begin{eqnarray*}
\lefteqn{(1+g_i + \ldots + g_i^{p-1})(\omega_3)}\\
&&= \omega_3 + (\omega_3 + \alpha(g_i) \omega_2 + \beta(g_i) \omega_1) + \ldots +
(\omega_3+  \alpha(g_i^{p-1}) \omega_2 + \beta(g_i^{p-1})\omega_1 )\\
&&=p \omega_3+ \binom{p}{2} \alpha(g_i) \omega_2 +
\left(\binom{p}{2} \beta(g_i) + \binom{p}{3} 2 \alpha(g_i)^2 \right) \omega_1
\end{eqnarray*}
and hence
\[(1+g_i + \ldots + g_i^{p-1})\omega_3 = 0\]
since $\mathrm{char}(k) = p > 3$; similarly we have $(1+ g_i+
\ldots + g_i^{p-1}) \omega_2 = 0$ and $(1+g_i+ \ldots
+g_i^{p-1}) \omega_1 = 0$. Therefore $e((\tilde{x},0)) = 0$ and
$e(x) = e((0,\hat{x}))$. We may assume that $\hat{x} = (c_{ij}
\omega_3)_{i < j}$ with some $c_{ij} \in k$ because $e$ maps
any element of $F_2 C_3$ into $\mathrm{im}(\partial_4)$. Then
$e((0, \hat{x})) = (a_1 \omega_2 + b_1 \omega_1, \ldots, a_s
\omega_2 + b_s \omega_1)$ where $a_1, \ldots, a_s$ and $b_1,
\ldots, b_s$ are given as follows. Let $M(\alpha)$ denote the
matrix whose rows are indexed by $k=1, \ldots, r$, whose
columns are indexed by the pairs $(i,j)$ with $i < j$ and whose
entry at the place $(k,(i,j))$ is equal to
\[\begin{cases}
\alpha(g_j) & \textrm{if } k=i\\
-\alpha(g_i) & \textrm{if } k=j\\
0 & \textrm{else}.
\end{cases}\]
Similarly the matrix $M(\beta)$ is defined. Then
\[\left(\begin{array}{c}a_1 \\ \vdots \\ a_s \end{array}\right)
= M(\alpha)((c_{ij})_{i<j}) \quad \textrm{ and } \quad
\left(\begin{array}{c} b_1 \\ \vdots \\ b_s \end{array}\right)
= M(\beta)((c_{ij})_{i < j}).\] We know that $a_1 = \ldots =
a_s =0$ because $X \in \mathrm{ker}(\partial_2)$. As
$\mathrm{im}(\partial_4)$ is equal to the kernel of the map
\[\bigoplus_{i=1}^s V_1 \rightarrow k, \quad (d_1 \omega_1,
\ldots, d_s \omega_1) \mapsto  \alpha(g_1)d_1+ \ldots +
\alpha(g_s)d_s\] (see above) we need to show that
\[\alpha(g_1)b_1 + \ldots + \alpha(g_s)b_s = 0.\]
For all $i < j$ the $(i,j)$-component of both the vectors
\[(\alpha(g_1), \ldots, \alpha(g_s)) M(\beta) \quad \textrm{ and } \quad (-\beta(g_1),
\ldots, - \beta(g_s)) M(\alpha).\] is equal to $\alpha(g_i) \beta(g_j) -\alpha(g_j)\beta(g_i)$. Hence
\[(\alpha(g_1), \ldots, \alpha(g_s)) M(\beta) = (-\beta(g_1),
\ldots, - \beta(g_s)) M(\alpha).\] Therefore we have
\begin{eqnarray*}
\lefteqn{\alpha(g_1) b_1 + \ldots + \alpha(g_s) b_s = (\alpha(g_1), \ldots, \alpha(g_s))
\left(\begin{array}{c} b_1 \\ \vdots \\ b_s \end{array}\right)}\\
&=& (\alpha(g_1), \ldots, \alpha(g_s)) M(\beta) ((c_{ij})_{i<j})\\
&=& (-\beta(g_1), \ldots, -\beta(g_s)) M(\alpha) ((c_{ij})_{i<j})\\
&=& (-\beta(g_1), \ldots, -\beta(g_s)) \left(\begin{array}{c} a_1 \\ \vdots \\ a_s \end{array} \right)
= (-\beta(g_1), \ldots, -\beta(g_s)) \left(\begin{array}{c} 0 \\ \vdots \\ 0 \end{array}\right) = 0,
\end{eqnarray*}
as desired.\\
We now turn to the case $p=3$. The above proof shows that the
first five dimensions are the same as in the case $p >3$.
However the (final) dimension $\dim_k E^3_{1,0}$ is equal to
$0$ in the case $p=3$, as we are going to prove now. As above
we have
\[ \dim_k E^2_{1,0} = \dim_k (\mathrm{coker}(\partial_4)) =1.\]
In order to prove that $\dim_k E^3_{1,0} = 0$ we will show that
the differential
\[\partial: E^2_{3,-1} = \ker(\partial_2) \rightarrow
\mathrm{coker}(\partial_4) = E^2_{1,0}\] is surjective. Using
the same calculation as in the case $p>3$ we obtain
\[(1+g_1 +g_1^2) (\omega_3) = 2 \alpha(g_1)^2 \omega_1.\]
Since $\alpha(g_1) \not= 0$, the tuple
$(2\alpha(g_1)^2\omega_1, 0, \ldots, 0)$ does not lie in the
image of~$\partial_4$ (use the same reasoning as in the case $p
>3$). Hence the differential $\partial$ is surjective and we
obtain
\[\dim_k E^3_{1,0} = 0,\]
as claimed above.\\
We finally prove Lemma~\ref{homologicalalgebra} in the case
$p=2$.  We may and will assume that not only $\alpha(g_1) \not=
0$ but $\alpha(g_i) \not= 0$ for all $i=1, \ldots, s$: if
$\alpha(g_i) = 0$ for some $i >1$, we replace $g_i$ by $g_i
g_1$. For brevity, we will write just $\alpha_i$ and $\beta_i$
for $\alpha(g_i)$ and $\beta(g_i)$, respectively. As $p=2$, the
map $\beta$ is a homomorphism as well. Also, the group $G$ acts
trivially on both $\omega_1$ and $\omega_2$ and hence on $V_2$.
Moreover the norm element $1+g_i$ is equal to $g_i -1$ and we
have
\[(g_i-1) (\omega_3) = \alpha_i \omega_2 + \beta_i
\omega_1\] for all $i=1, \ldots, s$. In particular all
differentials in $F_2 C.$ and $C./F_2 C.$ are zero and the long
exact sequence associated with the short exact sequence of
complexes
\[0 \rightarrow F_2 C. \rightarrow C. \rightarrow C./F_2 C.
\rightarrow 0\] looks as follows:
\[\xymatrix{&& \ldots \ar[r] & \left(\oplusm_{i=1}^s
V_3/V_2\right) \bigoplus \left(\oplusm_{i < j} V_3/V_2\right)
\ar `/16pt[r]  `[l]  `[llld]_{{\partial}} `[r] [dll]    &\\
& \oplusm_{i=1}^s V_2 \ar[r] & H_1(C.) \ar[r] & \oplusm_{i=1}^s V_3/V_2
\ar `/12pt[r]  `[l]  `[llld] `[r] [dll]    &\\
& V_2 \ar[r] & H_0(C.) \ar[r] & V_3/V_2 \ar[r] & 0. }\] Let
${\partial}$ denote the connecting homomorphism as indicated
above. When restricted to the direct sum $\bigoplus_{i=1}^s
V_3/V_2$ the map ${\partial}$ is the direct sum of the maps
$V_3/V_2 \rightarrow V_2$, $a\bar{\omega}_3 \mapsto a
\left(\alpha_i \omega_2 + \beta_i \omega_1\right)$, $i=1,
\ldots, s$. When restricted to the direct summand indexed by
the pair $(i,j)$, the map $\partial$ sends the basis element
$\bar{\omega}_3$ of $V_3/V_2$ to the tuple of
$\bigoplus_{i=1}^s V_2$ whose $i^\mathrm{th}$ component is
equal to $\alpha_j \omega_2 + \beta_j \omega_1$, whose
$j^\mathrm{th}$ component is equal to $\alpha_i \omega_2 +
\beta_i \omega_1$ and whose other components are all equal to
$0$. In particular the dimension of the image of~$\partial$ is
equal to to rank of the matrix whose columns are indexed by $1,
\ldots, s$ and the pairs $(i,j)$ for $i<j$, whose rows are
indexed by $1, \ldots, 2s$ and whose columns indexed by $i$ and
$(i,j)$ are equal to
\[(0, \ldots, 0, \alpha_i, 0, \ldots, 0|0, \ldots, 0, \beta_i,
0, \ldots, 0)^T\] and
\[(0, \ldots, 0, \alpha_j, 0, \ldots, 0, \alpha_i, 0, \ldots,
0|0, \ldots, 0, \beta_j, 0, \ldots, 0, \beta_i, 0, \ldots,
0)^T,\] respectively. Let $R_1, \ldots, R_{2s}$ denote the rows
of this matrix. As $\alpha_1, \ldots, \alpha_s$ are non-zero,
the rank does not change if we replace the row $R_{s+i}$ by
$\alpha_i R_{s+i} + \beta_i R_i$ for $i=1, \ldots, s$. Then the
columns indexed by $i$ and $(i,j)$ are equal to
\[(0, \ldots, 0, \alpha_i, 0, \ldots, 0|0, \ldots, 0)^T\]
and \begin{align*}(0, \ldots, 0, \alpha_j, 0, \ldots, 0,
&\alpha_i, 0, \ldots, 0|\\&0, \ldots, 0, \alpha_i\beta_j +
\alpha_j \beta_i, 0, \ldots, 0, \alpha_i\beta_j +\alpha_j
\beta_i, 0, \ldots, 0)^T,\end{align*} respectively. In
particular the ($s\times s$)-submatrix in the lower left-hand
corner is the zero matrix and the ($s \times s$)-submatrix in
the upper left-hand corner is the diagonal matrix with diagonal
entries $\alpha_1, \ldots, \alpha_s$ and has rank $s$. If
$\beta$ is a $k$-multiple of $\alpha$ the submatrix in the the
lower right-hand corner is the zero-matrix as well (or in fact
the empty matrix if $s=1$) and the rank of the total matrix is
$s$. Otherwise, we claim that the rank of the submatrix in the
lower right-hand corner is $s-1$, so the rank of the total
matrix is $2s-1$. From the above long exact sequence we then
obtain
\begin{eqnarray*}
\lefteqn{\dim_k H_0(G,V) - \dim_k H_1 (G,V) = \dim_k H_0(C.) - \dim_k H_1(C.)}\\
&=& \dim_k V_3/V_2 + \dim_k V_2 - \dim_k \left(\bigoplus_{i=1}^s V_3/V_2\right)
- \dim_k \mathrm{coker}({\partial})\\
&=&
\begin{cases}
1+2-s-(2s-s) = 3-2s & \textrm{if } \beta = c \alpha \textrm{ for some } c \in k,\\
1+2-s-(2s-(2s-1)) = 2-s & \textrm{else,}
\end{cases}
\end{eqnarray*}
as claimed above. Finally, to prove the above claim, we first
observe that the rank of the submatrix in the lower right-hand
corner is at most $s-1$ because the sum of all rows in this
submatrix is the zero row. To prove that the rank is at least
$s-1$ we will find $s-1$ columns of the submatrix that are
linearly independent. To this end, let $i_1 := 1$; as $\beta$
is not a $k$-multiple of $\alpha$, i.e.\ as the vector
$(\beta_1, \ldots, \beta_s)$ is not a $k$-multiple of the
vector $(\alpha_1, \ldots, \alpha_s)$, there exists $i_2 \in
\{2, \ldots, s\}$ such that $\alpha_1 \beta_{i_2} +
\alpha_{i_2} \beta_1 \not=0$; if $\alpha_{i_2} \beta_i +
\alpha_i \beta_{i_2} = 0$ for all $i \in \{1, \ldots, s\}
\backslash \{i_1, i_2\}$, let $i_3, \ldots, i_s$ run through
all indices in $\{1, \ldots, s\} \backslash \{i_1, i_2\}$;
otherwise choose $i_3 \in \{1, \ldots, s\} \backslash \{i_1,
i_2\}$ such that $\alpha_{i_2} \beta_{i_3} + \alpha_{i_3}
\beta_{i_2} \not= 0$; continuing this way we end up with
indices $i_1, \ldots, i_s$ such that $\{i_1, \ldots, i_s\} =
\{1, \ldots, s\}$ and such that there exists an index $l \in
\{1, \ldots, s-1\}$ with the property that $\alpha_{i_m}
\beta_{i_{m+1}} + \alpha_{i_{m+1}}\beta_{i_m}$ is not equal to
$0$ for $m=1, \ldots, l$ and is equal to $0$ for $m=l+1,
\ldots, s-1$; in particular we have $\alpha_{i_l} \beta_{i_m} +
\alpha_{i_m} \beta_{i_l} \not=0$ as well for $m=l+1, \ldots,
s-1$. We now consider the $s-1$ pairs $\widetilde{(i_1, i_2)},
\ldots, \widetilde{(i_{l-1}, i_l)}, \widetilde{(i_l, i_{l+1})},
\ldots, \widetilde{(i_l,i_s)}$ where the notation
$\widetilde{(i,j)}$ stands for the pair $(i,j)$ if $i<j$ and
for $(j,i)$ else. Then the $s-1$ columns $C_1, \ldots, C_{s-1}$
corresponding to these $s-1$ pairs are linearly independent: if
we have $a_1, \ldots, a_{s-1} \in k$ such that $a_1 C_1 +
\ldots + a_{s-1} C_{s-1} = 0$, then by successively looking at
the $i_1^\textrm{th}, \ldots, i_{l-1}^\textrm{th}$ component we
see that $a_1= 0, \ldots, a_{l-1}=0$ and by finally looking at
the $i_{l+1}^\textrm{th}, \ldots, i_s^\textrm{th}$ component we
see that $a_l=0, \ldots, a_{s-1} = 0$, as desired.
\end{proof}

\section{The $p$-rank representation.} \label{pranksection}

In this section we assume that $p >3$ and that $\pi$ is not
unramified. Let $D$ be a $G$-invariant effective canonical
divisor on $X$. We will see in Lemma~\ref{efecDG} below that
such a divisor always exists. As the divisor~$D$ is canonical
we have an isomorphism between the space
\[H_D := H^0(X, \Omega_X(D))\]
and the space $H^0(X, \Omega_X^{\otimes 2})$, this paper's main
object of study. As $D$ is effective we furthermore have the
decomposition
\[H_D = H_D^\mathrm{s} \oplus H_D^\mathrm{n}\]
where $H_D^\mathrm{s}$ and $H_D^\mathrm{n}$ are the spaces of
semisimple and nilpotent differentials with respect to the
Cartier operator on $H_D$, see \cite{Nak:85}, \cite{Stalder04}
or \cite{Subrao}. Note that this decomposition depends on the
actual divisor $D$ rather than just on its equivalence class;
this will become apparent in Theorem~\ref{semisimple} below,
for instance. We therefore work with the notation $H_D$ rather
than with $H^0(X, \Omega_X^{\otimes 2})$ in this section.

Since $D$ is $G$-invariant the above decomposition is a
decomposition of $k[G]$-modules. While little seems to be known
about the $k[G]$-module $H_D^\mathrm{n}$, the $k[G]$-module
$H_D^\mathrm{s}$ has been studied by various authors
(\cite{borne04}, \cite{Nak:85}, \cite{Stalder04}) and is called
the {\em $p$-rank representation}. As $G$ is a $p$-group the
only irreducible $k[G]$-module is the trivial representation
$k$ and has projective cover $k[G]$ \cite[15.6]{SerreLinear}.
We conclude that
\[H_D^\mathrm{s} \cong \mathrm{core}(H_D^\mathrm{s}) \oplus
k[G]^{b(G,D,k)}\] where $\mathrm{core}(H_D^\mathrm{s})$ denotes
the direct sum of non-projective indecomposable summands of
$H_D^\mathrm{s}$ (see \cite[Definition~2.3]{Stalder04}) and
$b(G,D,k)$ is called the {\em Borne invariant} corresponding to
$G$, $D$ and the trivial representation $k$ (see
\cite[Definition~5.1]{Stalder04}). The goal of this section is
to compute the multiplicity $b(G,D,k)$ when we impose further
conditions on the divisor~$D$, see Theorem~\ref{semisimple}
below. Via the isomorphisms $H_D \cong H^0(X, \Omega^{\otimes
2})$ and $H^0(X,\Omega^{\otimes 2})_G \cong H^1(G, \T_X)$,
Theorem~\ref{semisimple} gives us some information on the space
$H^1(G,\T_X)$, see Corollary~\ref{corollary1} below, or can be
used to derive some information on the nilpotent part
$H_D^\mathrm{n}$ if the $k[G]$-module structure of $H^0(X,
\Omega_X^{\otimes 2})$ is known, see
Corollary~\ref{directsummand}.

We begin with the following lemma which is of independent
interest and which we therefore formulate in a way that is
independent from the context of this paper.

\begin{lemma}\label{essential}
Let $Z$ be a connected smooth projective curve of genus at
least~$1$ over an algebraically closed field $k$, and let $S$
be a finite set of points on $Z$. Then there exists a global
non-zero holomorphic differential on $Z$ such that none of its
zeroes belongs to $S$.
\end{lemma}

\begin{proof}
If $g_Z=1$ then every non-zero holomorphic differential on $Z$
is non-vanishing \cite[III Proposition~1.5]{SilvermanArith},
and the result is obvious in
this case.\\
We therefore may and will assume that $g_Z \ge 2$. We will
prove Lemma~\ref{essential} by induction on $r:=|S|$. The case
$r=0$ is
trivial. So let $r \ge 1$ and let $P \in S$.\\
For the base step $r=1$ we need to show that
\[\mathrm{dim}_k H^0(Z,\Omega_Z) > \mathrm{dim}_k H^0(Z, \Omega_Z(-[P])).\]
By the Riemann-Roch theorem
\cite[Theorem~IV~1.3]{Hartshorne:77} the right-hand side is
equal to
\[(2g_Z-2-1)+1-g_Z+l([P]) = g_Z -2 + l([P]) = g_Z-1\]
because $1$ is a gap number in the sense of the Weierstrass Gap
Theorem \cite[Theorem~I.6.8]{StiBo} since $g_Z \ge 2$. Since
the left-hand side is equal to $g_Z$
this proves the base step $r=1$.\\
We now prove the inductive step. By the inductive hypothesis
there exists a global non-zero holomorphic differential
$\omega$ on $Z$ such that none of its zeroes belongs to
$S\backslash \{P\}$. Furthermore, by the case $r=1$ there
exists a global non-zero holomorphic differential $\phi$ on $Z$
that does not vanish at $P$. So for each $Q \in S$ at least one
of the two differentials $\omega$, $\phi$ does not vanish at
$Q$. Hence the sets
\[\{(\lambda,\mu) \in k^2: \lambda \omega + \mu \phi \textrm{
vanishes at } Q\}, \quad Q \in S, \] are one-dimensional subspaces of $k^2$.
Since $k$ is infinite we can avoid these finitely many lines
and hence find a pair $(\lambda, \mu) \in k^2$ such that none
of the zeroes of $\lambda \omega + \mu \phi$ belongs to $S$.
\end{proof}

For the following lemma we recall that the divisor of any
non-zero meromorphic differential is a canonical divisor
\cite[I.5.11]{StiBo}.

\begin{lemma} \label{efecDG}
There exists a $G$-invariant effective canonical divisor $D$ on
$X$ whose support contains $X_{ \mathrm{ram}}$. Moreover, if
$g_Y = 0$, we may choose $D$ in such a way so that its support
is equal to $X_{ \mathrm{ram}}$ and, if $g_Y \ge 1$, we may
choose $D$ of the form $D=\mathrm{div}(\pi^*\phi)$ where $\phi$
is a non-zero holomorphic differential on $Y$ whose zeroes do
not belong to $Y_{\mathrm{ram}}$.
\end{lemma}

\begin{proof}
If $\phi$ is a non-zero meromorphic differential on $Y$ then
$\pi^*\phi$ is a non-zero $G$-invariant meromorphic
differential on $X$ and its divisor $\mathrm{div}(\pi^* \phi)$
is hence a $G$-invariant canonical divisor on $X$. Furthermore
we have
\[
 \mathrm{div}(\pi^* \phi)=\pi^* \mathrm{div}(\phi)+R
\]
by \cite[Theorem~3.4.6]{StiBo}. We will choose $\phi$ in such a
way so that $D=\mathrm{div}(\pi^*\phi)$ is also effective and
its support
contains $X_{\mathrm{ram}}$.\\
If $g_Y>0$, then by Lemma~\ref{essential} there exists a
non-zero global holomorphic differential $\phi$ on~$Y$ whose
zeroes do not belong to $Y_\mathrm{ram}$. Then
$D=\mathrm{div}(\pi^*\phi)$ is certainly effective and its
support contains $X_\mathrm{ram}$. If $g_Y=0$ and $r=1$,  we
select a generator $x$ of the function field $K(Y)$ of $Y \cong
\PP^1_k$ such that $dx=-2[\pi(P_1)]$ and put
 $\phi:=dx$. Then we have
\begin{eqnarray*}
 \lefteqn{\mathrm{div}(\pi^* \phi) = \pi^* \mathrm{div}(\phi)+ R}\\
 & = & -2e_0(P_1) \sum_{P \in GP_1} [P]+R
 = \left(-2+ \sum_{i=2}^\infty (e_i(P_1)-1) \right)\sum_{P \in GP_1}[P].
\end{eqnarray*}
The Riemann-Hurwitz formula
\[
 2g_X-2=-2 |G| + \sum_{P \in GP_1} d(P_1)
\]
\cite[Corollary~IV~2.4]{Hartshorne:77} together with the
assumption that $g_X\geq 2$ tells us that $\pi$ is not weakly
ramified. Hence the coefficient $-2+\sum_{i=2}^\infty
(e_i(P_1)-1)$ is positive because $p>3$. Thus the divisor
$D=\mathrm{div}(\pi^*\phi)$ is effective and its support is
equal to $X_{\mathrm{ram}}$, as desired. If $g_Y=0$ and $r \geq
2$ we select a generator $x$ of $K(Y)$ such that $x(\pi(P_1))
\neq \infty$ and $x(\pi(P_2)) \neq \infty$ and put
$\phi=\frac{dx}{(x-x(\pi(P_1)))(x-x(\pi(P_2)))}$. Then we have
\[\mathrm{div}(\phi) = - [\pi(P_1)] - [\pi(P_2)]\]
and hence
\begin{eqnarray*} \label{eqdif2}
 \lefteqn{\mathrm{div}(\pi^*\phi) = - e_0(P_1) \sum_{P \in GP_1}[P]
 - e_0(P_2) \sum_{P\in GP_2} [P]
   +\sum_{j=1}^r  \sum_{P \in GP_j} d(P_j) [P] }\\
 &=& \sum_{j=1}^2 \left( -1 + \sum_{i=1}^{\infty} (e_i(P_j)-1) \right) \sum_{P \in GP_j}[P]+
 \sum_{j=3}^r  \sum_{P \in GP_j} d(P_j) [P].
\end{eqnarray*}
Again we see that $D=\mathrm{div}(\pi^*\phi)$ is effective and
its support is equal to $X_\mathrm{ram}$, as desired.
 \end{proof}

Let $\gamma_X$ and $\gamma_Y$ denote the $p$-ranks of $X$ and
$Y$, respectively. Furthermore, we recall that $E_\mathrm{red}$
denotes the reduced divisor for any effective divisor $E$.

\begin{theorem}\label{semisimple}
Let $D$ be a $G$-invariant effective canonical divisor on $X$
as in Lemma~\ref{efecDG}. Then the $k[G]$-module
$H_D^\mathrm{s}$ is free of rank
\[ b(G,D,k) = \begin{cases} \gamma_Y - 1 + r & \textrm{if } g_Y=0\\
\gamma_Y - 1 + r + \mathrm{deg}(\mathrm{div}(\phi)_\mathrm{red})& \textrm{if } g_Y \ge 1.
\end{cases}\]
\end{theorem}

\begin{proof}
By \cite[Lemma~2.5]{Subrao}, the $k[G]$-module $H_D^\mathrm{s}$
is the same as $H_{D_\mathrm{red}}^\mathrm{s}$. Furthermore, if
$S$ is any nonempty set of $Y$ containing $Y_\mathrm{ram}$ and
$E$ is the effective reduced $G$-invariant divisor on $X$
corresponding to $\pi^{-1}(S)$, then the semisimple part of
$H^0(X, \Omega_X(E))$ is a free $k[G]$-module of rank $\gamma_Y
-1 + |S|$ by a theorem of Nakajima \cite[Theorem~1]{Nak:85}. As
\[D_\mathrm{red} =
\begin{cases}
R_\mathrm{red} & \textrm{if } g_Y=0\\
\mathrm{div}(\pi^* \phi)_\mathrm{red} =
(\pi^* \mathrm{div}(\phi))_\mathrm{red} + R_\mathrm{red} & \textrm{if } g_Y \ge 1,
\end{cases}\]
the divisor $D_\mathrm{red}$ corresponds to
\[\begin{cases}
X_\mathrm{ram} = \pi^{-1}(Y_\mathrm{ram}) & \textrm{if } g_Y=0\\
\pi^{-1}(\mathrm{supp}(\mathrm{div}(\phi))) \sqcup X_\mathrm{ram}=
\pi^{-1}(\mathrm{supp}(\mathrm{div}(\phi)) \sqcup Y_\mathrm{ram}) & \textrm{if } g_Y \ge 1
\end{cases}\]
and Theorem~\ref{semisimple} follows from
\cite[Theorem~1]{Nak:85}.
\end{proof}

Theorem~\ref{semisimple} has the following possible
applications. Firstly, if $\dim_k (H^\mathrm{n}_D)_G$ and, in
the case $g_Y \ge 1$, also
$\mathrm{deg}(\mathrm{div}(\phi)_\mathrm{red})$ can be
computed, we can also compute the dimension $\dim_k H^1(G,
\cT_X) = \dim_k H^0(X, \Omega_X^{\otimes 2})_G$ of the tangent
space of the deformation functor associated with $G$ acting on
$X$:

\begin{cor}\label{corollary1}
 Let $D$ be as in Theorem~\ref{semisimple}. Then we have
\begin{align*}\dim_k &H^0(X, \Omega_X^{\otimes 2})_G\\
& =
 \begin{cases}
  \dim_k (H^\mathrm{n}_D)_G + \gamma_Y - 1 +r & \textrm{if } g_Y=0\\
  \dim_k (H^\mathrm{n}_D)_G + \gamma_Y - 1 +r + \mathrm{deg}(\mathrm{div}(\phi)_\mathrm{red})
  & \textrm{if } g_Y \ge 1.
 \end{cases}
\end{align*}
\end{cor}

\begin{proof}
 Obvious.
\end{proof}

For the second application we first introduce the following
definition. If a $k[G]$-module $M$ is isomorphic to a direct
sum $\bigoplus_{i=1}^l M_i ^{\oplus m_i}$ for some pairwise
non-isomorphic indecomposable $k[G]$-modules $M_1, \ldots, M_l$
and some natural numbers $m_1, \ldots, m_l$, we call $m_i$ the
multiplicity of $M_i$ in $M$. In cases the multiplicity
$m_{k[G]}$ of the regular representation $k[G]$ in the
$k[G]$-module $H^0(X, \Omega_X^{\otimes 2})$ is known, we can
also compute the multiplicity $n_{k[G]}$ of $k[G]$ in the
nilpotent part $H_D^\mathrm{n}$:

\begin{cor}\label{directsummand}
Let $D$ be as in Theorem~\ref{semisimple}. Then we have
\[n_{k[G]} = m_{k[G]} -
\begin{cases}
\gamma_Y - 1 + r & \textrm{if } g_Y=0\\
\gamma_Y - 1 + r + \mathrm{deg}(\mathrm{div}(\phi)_\mathrm{red}) &
\textrm{if } g_Y \ge 1.
\end{cases}\]
\end{cor}

\begin{proof}
Obvious.
\end{proof}

A formula for $m_{k[G]}$ is known if $G$ is cyclic
\cite[Theorem~7.23]{Borne06} or if G is elementary abelian \cite[Section~3]{karan}. For
instance, if $G$ is cyclic of order $p$, we can derive from
\cite[Theorem~1]{Nak:86} that
\[m_{k[G]} = 3g_Y-3 + \sum_{j=1}^r \left\lfloor
\frac{(N_j+2)(p-1)}{p}\right\rfloor \] where $N_j$ is the
highest (and single) jump in the lower ramification filtration of~$G(P_j)$.

\section{Appendix}\label{account}

This appendix gives an account of a structure theorem for
finite weakly ramified Galois extensions of local fields of
positive characteristic $p$ in the case the Galois group is a
$p$-group, see Proposition~\ref{structuretheorem} below. This
structure theorem is used at the end of the proof of
Proposition~\ref{projective} in the main part of this paper,
but only in the case $p=2$. It appears as (part of)
Proposition~1.4 in ~\cite{CK} and Proposition~1.18 in
\cite{CK2}; in this appendix we
give a self-contained and elementary proof. It
implies a very explicit and simple description of the action of
the Galois group on a local parameter, see
Proposition~\ref{structuretheorem}. In the situation of
Section~\ref{sec3} of the main part of this paper we finally
derive a certain feature of pole numbers, if $p=2$, see
Corollary~\ref{Weierstrass}.

Let $K$ be a local field of characteristic $p > 0$; i.e.\ the
field $K$ is complete with respect to a discrete valuation
$v_K: K^\times \twoheadrightarrow \ZZ$ and its residue field
$k$ is perfect; we assume that $k$ is contained in $K$.


\begin{lemma}\label{cyclicextension}
Let $L/K$ be a totally ramified Galois extension of degree $p$.
Then there exists an element $y \in L$ whose valuation is
coprime to $p$ and negative, say $-m$, such that $y^p-y \in K$
and $L=K(y)$. The greatest integer~$M$ such that the higher
ramification group $G_M$ of $L/K$ does not vanish is then equal
to $m$. 
\end{lemma}

\begin{proof}
Let $\sigma$ be a generator of the Galois group
$G=\mathrm{Gal}(L/K)$. By the classical Artin-Schreier Theorem
there exist elements $x \in K$ and $y \in L$ such that $L=K(y)$
and $y^p-y = x$. We first show that the valuation $v_K(x)$ of
$x$ is negative and that $v_L(y) = v_K(x)$. Suppose that
$v_K(x) \ge 0$. Then also $v_L(y) \ge 0$ and, denoting the
residue classes of $y$ and $x$ in $k$ by $\bar{y}$ and
$\bar{x}$, respectively, we obtain
\[(y-\bar{y})^p - (y - \bar{y}) = (y^p -y) - \overline{y^p-y} =
x - \bar{x}.\] In particular we have $\sigma(y-\bar{y}) =
(y-\bar{y}) +c$ for some $c \in \FF_p^\times$ and
\[0 < v_L(y-\bar{y}) = v_L(\sigma(y-\bar{y})) = v_L((y-\bar{y})
+c) = 0,\] which is a contradiction. Hence we have $v_K(x) <0$,
$v_L(y) <0$ and
\[pv_L(y) = v_L(y^p) = v_L(y^p-y) = v_L(x) = p v_K(x),\]
i.e.\ $v_L(y) = v_K(x)$. Let $m:= - v_L(y) = - v_K(x)$. If $p$
divides $m$, i.e.\ $m=lp$ for some $l \in \NN$, we write
\[x= \frac{u_0}{s^{lp}} + \frac{u_1}{s^{lp-1}} +
\frac{u_2}{s^{lp-2}} + \ldots\] with some $u_0, u_1, u_2,
\ldots \in k$ and some local parameter $s$ of $K$. Furthermore,
as $k$ is perfect, there is a $v_0 \in k$ such that $v_0^p =
u_0$. We now consider
\[\tilde{y} := y - \frac{v_0}{s^l} \in L \quad \textrm{and}
\quad \tilde{x} := x - \frac{v_0^p}{s^{lp}} + \frac{v_0}{s^l}
\in K.\] Then we have $L=K(\tilde{y})$ and
\[\tilde{y}^p - \tilde{y} = \left(y-\frac{v_0}{s^l}\right)^p - \left(y -
\frac{v_0}{s^l}\right) = (y^p - y) - \frac{v_0^p}{s^{lp}} +
\frac{v_0}{s^l} = \tilde{x}.\] As above we obtain
$v_L(\tilde{y}) = v_K(\tilde{x}) < 0$; furthermore we have
\begin{align*}
v_K(\tilde{x}) = v_K&\left(x-\frac{v_0^p}{s^{lp}} +
\frac{v_0}{s^l}\right) \\&= v_K \left(\left(\frac{u_1}{s^{lp-1}} +
\frac{u_2}{s^{lp-2}} + \ldots\right) + \frac{v_0}{s^l}\right)
> -lp = -m = v_K(x).
\end{align*}
Continuing this way (if necessary) we obtain $x \in K$ and $y
\in L$ such $L=K(y)$, $y^p-y=x$ and $p$ does not divide $m = -
v_L(y) = - v_K(x) >0$, as claimed.\\
Let $r, l \in \ZZ$ such that $rp +lm=1$. Let $s$ be a local
parameter of $K$ and put $t:= s^r y^{-l} \in L$. Then we have
$v_L(t) = rp + (-l)(-m) = 1$, i.e.\ $t$ is a local parameter of
$L$. Furthermore we have
\[\sigma(t) = \sigma(s^r y^{-l}) = s^r \sigma(y)^{-l} = s^r
(y-c)^{-l} \quad \textrm{for some } c \in \FF_p^\times, \] hence
\[\sigma(t) = s^r \left(\frac{y^{-1}}{1-cy^{-1}}\right)^l = t(1+ lcy^{-1} +
\ldots)\] and finally
\[v_L(\sigma(t) - t) = v_L(t(lcy^{-1})) = 1 +m,\]
as was to be shown.
\end{proof}

\begin{prop}\label{structuretheorem}
Let $L/K$ be a finite weakly ramified Galois extension of local
fields such that the residue field $k$ of $K$ is algebraically
closed and the Galois group $G$ is a $p$-group. Then there
exist $t_0 \in K$, $t_1, \ldots, t_n \in L$ and $c_1, \ldots,
c_{n-1} \in k^\times$ such that $L=K(t_1, \ldots, t_n)$, $t:=
t_n$ is a local parameter of~$L$, $t_1^{-p} - t_1^{-1} =
t_0^{-1}$ and $t_i^{-p} - t_i^{-1} = c_{i-1} t_{i-1}^{-1}$ for
$i=2, \ldots, n$. Furthermore for each $g \in G$ there is an $a
\in k$ (depending on $g$) such that
\[g(t) = \frac{t}{1-at} = t(1+at + a^2 t^2 + \ldots).\]
\end{prop}

\begin{proof}
As $L/K$ is weakly ramified, the Galois group $G$ is in fact an
elementary abelian $p$-group, i.e.\ $G= C_1 \times \ldots
\times C_n$ for some subgroups $C_1, \ldots, C_n$ of $G$ of
order $p$. We proceed by induction on $n$. If $n=0$ there is
nothing to prove. So let $n \ge 1$. We write just $C$ for
$C_n$. By the inductive hypothesis there exist $c_1, \ldots,
c_{n-2} \in k^\times$, $t_0 \in K$ and $t_1, \ldots, t_{n-1}$
in the fixed field $L^C$  such that $L^C = K(t_1, \ldots,
t_{n-1})$, the element $t_{n-1}$ is a local parameter of $L^C$,
$t_1^{-p} - t_1^{-1} = t_0$ and $t_i^{-p} - t_i^{-1} = c_{i-1}
t_{i-1}^{-1}$ for $i=2, \ldots, n-1$. By
Lemma~\ref{cyclicextension} there exist local parameters $s$
and $t$ of $L^C$ and $L$, respectively, such that $L=L^C(t)$
and $t^{-p} - t^{-1} = s^{-1}$. If $n=1$ we redefine $t_0$ to
be equal to $s$ and put $t_1:= t$. If $n >1$ we write $s^{-1} =
c_{n-1} t_{n-1}^{-1} + e$ for some $c_{n-1} \in k^\times$ and
some $e$ in the valuation ring $\cO_{L^C}$ of $L^C$. As $k$ is
algebraically closed and $\cO_{L^C}$ is Henselian we can find
$f \in \cO_{L^C}$ such that $f^p - f = e$. We finally define
$t_n$ to be the inverse of $t^{-1} - f$ and obtain $t_n^{-p} -
t_n^{-1} = (t^{-p} - t^{-1}) - (f^p - f) = c_{n-1}
t_{n-1}^{-1}$, as desired. To prove the last assertion we fix
$g \in G$. By the inductive hypothesis there exists $b \in k$
such that $g(t_{n-1})= \frac{t_{n-1}}{1-bt_{n-1}}$, i.e.\
$g(t_{n-1}^{-1}) = t_{n-1}^{-1} - b$. Then we have
\[g(t^{-1})^p - g(t^{-1}) = g (t_n^{-p} - t_n^{-1}) = g
(c_{n-1} t_{n-1}^{-1}) = c_{n-1}t_{n-1}^{-1} - c_{n-1} b.\]
Hence we have $g(t^{-1}) = t^{-1} - a$, i.e.\ $g(t) =
\frac{t}{1-at}$, for some Artin-Schreier root~$a$ of
$c_{n-1}b$, as desired.
\end{proof}

We now return to the situation considered in the main part of
this paper. We recall that a natural number $m$ is called a
{\em pole number} at a point $P \in X$ if there exists a
meromorphic function $f$ on $X$ whose pole order at $P$ is $m$
and which is holomorphic everywhere else. The Riemann-Roch
theorem (see \cite[Theorem~IV~1.3]{Hartshorne:77}) implies that
every integer $m > 2g_X-2$ is a pole number, see also
\cite[Proposition~1.6.6]{StiBo}.

\begin{cor}\label{Weierstrass}
Suppose that $p=2$ and that $\pi: X \rightarrow Y$ is weakly
ramified. Let $P \in X_\mathrm{ram}$ such that $G(P)$ is not
cyclic. Then the smallest odd pole number $m$ at $P$ is
congruent to $1$ modulo $4$ and $m-1$ is a pole number as well.
\end{cor}

The following proof is a refinement of some arguments given in
\cite{Ko:2008}.

\begin{proof}
Let $\hat{\cO}_{X,P}$ denote the completion of the local ring
$\cO_{X,P}$ at $P$. Let $t \in \hat{\cO}_{X,P}$ be a local
parameter at $P$ and let the maps $\alpha$ and $\beta$ from
$G(P)$ to $k$ be defined by the congruences
\[g(t) \equiv t + \alpha(g) t^2 + \beta(g) t^3 \quad
\mathrm{mod} \quad (t^4),\] $g \in G(P)$. We recall that $\alpha$ is
a monomorphism, see the proof of
Proposition~\ref{grouphomology}. We first observe that changing
the local parameter $t$ by multiplication with a unit $u \in
\hat{\cO}_{X,P}^\times$ amounts to multiplying $\alpha$ and
$\beta$ with $\bar{u}^{-1}$ and $\bar{u}^{-2}$, respectively,
where $\bar{u}$ denotes the residue class of $u$ in $k^\times$.
From Proposition~\ref{structuretheorem} we therefore conclude
that $\beta = \alpha^2$ no matter which local parameter $t$ we
choose. We will prove below that we can choose $t$ in such a
way that $\beta = \frac{m+1}{2} \alpha^2 + c\alpha$ for some
constant $c \in k$. As $G(P)$ is not cyclic we therefore obtain
$\frac{m+1}{2} = 1$ in $\FF_2$ (and $c=0$), i.e.\ $m$ is
congruent to $1$ modulo $4$, as stated. \\
Let $f$ be a function in $H^0(X,\cO_X(m[P]))$ whose pole order
at $P$ is $m$. Then the function $f^{-1}$ has a zero at $P$ of
order $m$. As $\hat{\cO}_{X,P}$ is Henselian and $m$ is odd
there exists an $m^\textrm{th}$ root $t$ of $f^{-1}$ in
$\hat{\cO}_{X,P}$. Obviously, $t$ is a local parameter at $P$.
Let $m >m_1> \ldots
> m_s =0$ be the sequence of pole numbers at $P$ smaller than
$m$. Note that $m_1, \ldots, m_s$ are even by assumption. Let
$f_1, f_2, \ldots, f_s$ be functions in $H^0(X,\cO_X(m[P]))$ of
pole orders $m_1, \ldots, m_s$, respectively. Then there exist
some units $u_1, \ldots, u_s \in \hat{\cO}_{X,P}^\times$ such
that $f_1 = \frac{u_1}{t^{m_1}}, \ldots, f_s=
\frac{u_s}{t^{m_s}}$; without loss of generality we may assume
that $\bar{u}_1 = \ldots = \bar{u}_s = 1$ in $k^\times$. As
$m[P]$ is a $G(P)$-invariant divisor and $f, f_1, \ldots, f_s$
form a basis of $H^0(X,\cO_X([mP]))$ (see \cite[p.~34]{StiBo}),
there exist maps $c_0, \ldots, c_s$ from $G(P)$ to $k$ such
that
\[g(f) = c_0(g)f + c_1(g)f_1 + \ldots + c_s(g)f_s\]
for all $g \in G(P)$; i.e.~we have
\begin{equation}\label{actionofgonf}
g\left(\frac{1}{t^m}\right) = g(f) = \frac{c_0(g)}{t^m} + \frac{c_1(g) u_1}{t^{m_1}} + \frac{c_2(g)u_2}{t^{m_2}}
+ \ldots + \frac{c_s(g) u_s}{t^{m_s}}.
\end{equation}
On the other hand, using the geometric series and the binomial
theorem we derive the following congruence for each $g \in
G(P)$:
\begin{eqnarray*}
\lefteqn{ g\left(\frac{1}{t^m}\right) = g(t)^{-m}}\\
& \equiv & t^{-m} \left(1+\alpha(g)t +\beta(g)t^2\right)^{-m}\\
& \equiv & t^{-m} \left(1-\alpha(g)t -\beta(g)t^2 +
\alpha^2(g)t^2\right)^m \\
& \equiv & t^{-m} \left(1+m \left(-\alpha(g)t - \beta(g)t^2 +
\alpha^2(g) t^2\right) + \binom{m}{2} \alpha^2(g)t^2 \right)\\
& \equiv & \frac{1}{t^m} + \frac{-m\alpha(g)}{t^{m-1}} +
\frac{-m\beta(g) + \binom{m+1}{2} \alpha^2(g)}{t^{m-2}} \quad
\mathrm{mod} \quad \left(\frac{1}{t^{m-3}}\right).
\end{eqnarray*}
Let $c \in k$ be defined by the congruence $u_1 \equiv 1 + ct$
mod~$(t^2)$. By comparing the coefficients in
(\ref{actionofgonf}) with the coefficients in this latter
congruence we obtain the following equalities for each $g \in
G(P)$: $c_0(g) =1$; $m_1 = m-1$ (because $G(P)$ is non-trivial
and $\alpha$ is injective) and $c_1(g) = -m\alpha(g)$; $c_1(g)
c = -m\beta(g) + \binom{m+1}{2} \alpha^2(g)$ (because $m_s <
\ldots < m_2 < m-2$). Now replacing $c_1(g)$ with $-m\alpha(g)$
in the latter equation and dividing by $m$ we obtain the
desired equations
\[\beta(g) = \frac{m+1}{2} \alpha^2(g) + c \alpha (g), \quad g \in G(P).\]
As $m-1 = m_1$, we have also proved that $m-1$ is a pole
number.

\end{proof}

\providecommand{\bysame}{\leavevmode\hbox to3em{\hrulefill}\thinspace}
\providecommand{\MR}{\relax\ifhmode\unskip\space\fi MR }
\providecommand{\MRhref}[2]{%
  \href{http://www.ams.org/mathscinet-getitem?mr=#1}{#2}
}
\providecommand{\href}[2]{#2}

\bigskip

School of Mathematics, University of Southampton, Highfield,
Southampton, SO17 1BJ, United Kingdom.

Department of Mathematics, National and Kapodistrian University
of Athens, Panepistimioupolis, GR-157 84, Athens Greece.

\end{document}